\documentclass[reqno]{amsart}

\usepackage{graphicx}
\usepackage[usenames,dvipsnames,svgnames,table]{xcolor}

\usepackage[charter]{mathdesign}

\usepackage{tikz}
\usetikzlibrary{backgrounds}
\usetikzlibrary{patterns,fadings}
\usetikzlibrary{arrows,decorations.pathmorphing}
\usetikzlibrary{calc}
\definecolor{light-gray}{gray}{0.95}
\usepackage{float}
\usepackage[colorlinks=true,linkcolor=blue,citecolor=magenta]{hyperref}

\newtheorem{theorem}{Theorem}[section]
\newtheorem{lemma}[theorem]{Lemma}
\newtheorem{proposition}[theorem]{Proposition}

\newtheorem{remark}[theorem]{Remark}
\newtheorem{definition}{Definition}[section]

\numberwithin{equation}{section}

\newcommand{\<}{\langle}
\renewcommand{\>}{\rangle}

\renewcommand{\epsilon}{\varepsilon}

\def\centerarc[#1](#2)(#3:#4:#5){\draw[#1] ($(#2)+({#5*cos(#3)},{#5*sin(#3)})$) arc (#3:#4:#5);}

\let\oldtocsection=\tocsection
\let\oldtocsubsection=\tocsubsection
\let\oldtocsubsubsection=\tocsubsubsection
\renewcommand{\tocsection}[2]{\hspace{0em}\oldtocsection{#1}{#2}}
\renewcommand{\tocsubsection}[2]{\hspace{1em}\oldtocsubsection{#1}{#2}}
\renewcommand{\tocsubsubsection}[2]{\hspace{2em}\oldtocsubsubsection{#1}{#2}}
\DeclareRobustCommand{\SkipTocEntry}[5]{}
\newcommand{\mc}[1]{{\mathcal #1}}

\newcommand{\bb}[1]{{\mathbb #1}}

\newcommand{\cadlag}{c\`adl\`ag~}

\DeclareMathOperator{\sgn}{sgn}

\title{}
\date{}
\author{}

\begin{document}

\title[SBE from long range exclusion interactions]{Stochastic Burgers equation from long range \\exclusion interactions}

\author{Patr\'{\i}cia Gon\c{c}alves}
\address{\noindent Center for Mathematical Analysis,  Geometry and Dynamical Systems \\
Instituto Superior T\'ecnico, Universidade de Lisboa \\
Av. Rovisco Pais, 1049-001 Lisboa, Portugal.}
\email{patricia.goncalves@math.tecnico.ulisboa.pt}

\author{Milton Jara}
\address{Instituto de Matem\'atica Pura e Aplicada, Estrada Dona Castorina 110, 22460-320  Rio de Janeiro, Brazil.}  \email{mjara@impa.br}


\begin{abstract}
We consider one-dimensional  exclusion processes  with long jumps given by a transition probability of the form $p_n(\cdot)=s(\cdot)+\gamma_na(\cdot)$, such that its symmetric part $s(\cdot)$ is irreducible with  finite variance and  its antisymmetric part is absolutely bounded by $s(\cdot).$ We prove that under diffusive time scaling  and strength of asymmetry  $\sqrt n \gamma_n \to_{n\to\infty} b\neq 0$, the equilibrium density fluctuations are given by the unique energy solution of the stochastic Burgers equation. 
\end{abstract}

\maketitle

\section{Introduction}

Scaling limits of one-dimensional, stochastic interface models and their relation with the Kardar-Parisi-Zhang  (KPZ) \cite{KPZ} equation 
\[
\partial_t \mc X_t =  \Delta \mc X_t + \lambda \big(\nabla \mc X_t)^2 +  \xi
\]
have been the subject of intense study in the mathematical physics community during the last few years. In this equation, $\xi$ stands for a standard, space-time white noise. One line of research corresponds to what is known as the {\em weak universality} of the KPZ equation. Roughly speaking, the KPZ equation should arise as the scaling limit of any one-dimensional, non-reversible family of stochastic interface models with local interactions on which there is a natural way to scale down the asymmetry of the model as the scale grows. 

From the point of view of interacting particle systems, it is more natural to look at the derivative of the interface. The stochastic process obtained in this way corresponds to a non-reversible particle system with a {\em conservation law}, most of the times the number of particles of the system. This conservation law is equivalent to the local character of the dynamics in the interface model. The limiting equation will be the stochastic Burgers equation
\begin{equation}
\label{frambuesa}
\partial_t \mc Y_t = \tfrac{1}{2} \sigma^2 \Delta \mc Y_t + \lambda \nabla \mc Y_t^2 + \nu \nabla \xi,
\end{equation}
where $\xi$ is a space-time white noise and $\sigma^2$, $\lambda$, $\nu$ are model-dependent constants. This equation is the one satisfied by the space derivative of the solution of the KPZ equation.

The first mathematically rigorous work showing the convergence of the density fluctuations of an interface growth model is the one of Bertini and Giacomin \cite{BerGia}. The authors considered the weakly asymmetric simple exclusion process (WASEP) and they proved that its density fluctuations converge to the so-called {\em Cole-Hopf} solution of \eqref{frambuesa}. Their result holds for a large class of initial conditions and it was also the first article where a proper mathematically rigorous definition of the KPZ/Burgers equation was given. Interpreted in terms of interface models, in \cite{BerGia} the authors proved that the scaling limit of {\em current fluctuations} in the WASEP is given by the KPZ equation.

The results of \cite{BerGia} heavily rely on the so-called {\em Gartner transform} (see \cite{DitGar,Gar}), which is a discrete version of the Cole-Hopf transform that effectively linearizes the asymmetric simple exclusion process (ASEP). This linearization is a blueprint of the richer integrable structure of the ASEP. This integrable structure was discovered by relating the ASEP to the Bethe Ansatz \cite{Sch}, \cite{TraWid}. 
The theory of MacDonald processes \cite{BorCor} allows to discover various other systems with rich integrable structure, as well as to extend the integrable structure of ASEP \cite{BorCorSas}.

In order to prove convergence of density fluctuations to solutions of \eqref{frambuesa}, there are two well developed strategies, which are very different and complementary in nature. One strategy is to consider a model with enough structure to allow a linearizing transformation, similar to Gartner transform. If this strategy is successful, the original method of Bertini and Giacomin \cite{BerGia} should work, modulo model-dependent technical points. This strategy has been performed by instance in \cite{AmiCorQua}, \cite{DemTsa}, \cite{CorTsa}, \cite{CorSheTsa}, \cite{Lab}. When this strategy works, the underlying integrable structure usually allows to obtain fine details of the solutions of the KPZ equation starting from special initial conditions, as well as to allow very general initial data.

The second strategy relies on the concept of {\em energy solutions} of the KPZ/Burgers equation \eqref{frambuesa}, introduced in \cite{GonJar}. Conditioned on the uniqueness of stationary solutions of this equation, the authors  proved in  \cite{GonJar} convergence of the density fluctuations to this stationary solution for a large class of weakly asymmetric  stochastic models. In particular, no integrability assumptions are made. The convergence of interface fluctuations of the associated growth models is also obtained as a consequence of the convergence of current fluctuations. Uniqueness of stationary energy solutions of the KPZ/Burgers equations as defined in \cite{GubJar} were proved to hold on the torus in \cite{GubPer3}  and on the real line in  \cite{DieGubPer}, closing the gap in the proof of convergence of \cite{GonJar}. In \cite{GonJarSet}, the proof of convergence was generalized to include particle systems without product invariant measures and to include non-stationary initial data for the KPZ/Burgers equation. This strategy has been applied to show convergence to the stationary energy  solution of \eqref{frambuesa} for a variety of models in \cite{BloGonSim}, \cite{GonJarSim}, \cite{GubPer3}, \cite{FraGonSim}, \cite{DieGubPer}.

These two strategies to prove convergence to \eqref{frambuesa} are complementary in the following sense. On one hand, the main drawback of the method based on the Gartner-Cole-Hopf transform is that it only works for models satisfying very specific conditions of algebraic and/or combinatoric type. Therefore, even perturbations vanishing in the limit will prevent the use of this transform. For the method based on energy solutions, the set of assumptions required for the model are basically the same required to  prove convergence of the symmetrized model to the Ornstein-Uhlenbeck equation
\[
\partial_t Y_t = \tfrac{1}{2} \sigma^2 \Delta \mc Y_t + \nu \nabla \xi.
\]
On the other hand, the main drawback of the method based on energy solutions is that convergence can be proved only for {\em almost stationary} initial conditions, in the sense that the initial condition of \eqref{frambuesa} should have bounded entropy with respect to the stationary one.

Let us mention here that a satisfactory theory of solutions of \eqref{frambuesa} has been elusive for a long time. The main difficulty comes from the fact that solutions of the linear part of \eqref{frambuesa} are locally absolutely continuous with respect to spatial white noise, and scaling arguments show that the nonlinear term should not modify the local regularity of the solution. This means that solutions are distribution-valued, and it is far from clear how to properly define the square of the solution appearing in the nonlinearity.
In a pair of breakthrough works \cite{Hai1} and  \cite{Hai2} M. Hairer developed the theory of {\em regularity structures} to deal exactly with this kind of difficulties. In \cite{Hai1} he formulated a notion of solution of \eqref{frambuesa} that gives uniqueness and he proved that this solution coincides with the Cole-Hopf solution introduded in \cite{BerGia}. However, the notion of regularity structures is yet to be successfully adapted to the context of interacting particle systems. The notion of energy solutions, based on martingale problems, is close in spirit to the classical martingale characterization of Markov processes of Stroock and Varadhan \cite{StrVar}, and the method of proof is inspired by the theory of hydrodynamic limits of particle systems. Therefore it fits very well in the context of interacting particle systems.

The exclusion process, as introduced by Spitzer \cite{Spi} (see the classical book of Liggett \cite{Lig} for a detailed exposition of this process) has been one of the most studied interacting particle systems. Despite its simplicity, it is remarkable that many phenomena in non-equilibrium statistical physics can be modeled using the exclusion process.
As can be seen from the list of references already mentioned, the KPZ/Burgers equation is not an exception to this rule. The seminal article \cite{BerGia} derived the KPZ/Burgers equation from the WASEP, and \cite{AmiCorQua} took one step further by computing the probability laws of some observables of interest of the KPZ equation using the convergence of WASEP to the KPZ equation. In \cite{DemTsa}, the proof of \cite{BerGia} was extended to exclusion processes with range at most $3$. Looking at the discussion in \cite{DemTsa}, it becomes more or less clear that their result is optimal in the sense that no exclusion process with range greater or equal to $4$ should be linearized by a properly chosen Gartner transform. The aim of this article is to treat the exclusion process with arbitrary range as a test case for the method of energy solutions. As already pointed out in \cite{DemTsa}, \cite{GonJarSet}, \cite{GonJarJSP}, the parameters in \eqref{frambuesa} depend on the first moment of the jump rate of the exclusion process and on the second moment of the symmetric part of this jump rate. Therefore, it is reasonable to consider a weakly asymmetric jump rate of the form
\begin{equation}
\label{platano}
p_n(z) = s(z) + \gamma_n a(z),
\end{equation}
where $s(\cdot)$ is a symmetric jump rate in $\bb Z$, $\gamma$ satisfies $\gamma_n \sqrt{n} \to b \neq 0$ and $a(\cdot)$ is an antisymmetric function on $\bb Z$, and to assume that 
\begin{equation}
\tag{H1}
\sigma^2 = \sum_{z \in \bb Z} z^2 s(z) <+\infty.
\end{equation}
In \cite{DemTsa}, it is shown that the Gartner transform linearizes the problem if and only if $s(z)$ has range at most $3$ and $a(\cdot)$ is chosen as a particular function of $s(\cdot)$ (the one-parameter family of rates mentioned in \cite{DemTsa} corresponds to the choice of $b$ here). Since we also want $p_n(\cdot)$ to be non-negative, it is necessary to assume that
\begin{equation}
\tag{H2}
\text{there exists a finite constant $C$ such that } |a(z)| \leq Cs(z) \text{ for any } z \in \bb Z.
\end{equation}
%%%
The main objective of this article is to prove Theorem \ref{t1}, which states that under these two assumptions on $p_n(\cdot)$, plus the irreducibility of $s(\cdot)$, the density fluctuations of the weakly asymmetric exclusion processes associated to $p_n(\cdot)$ starting from the stationary state converge to the stationary energy solution of \eqref{frambuesa}. Notice that \eqref{H1} is necessary to have a well-defined limit on the diffusive time scale, even for {\em symmetric} exclusion processes, and \eqref{H2} is necessary to have a well-defined process. Therefore, we say that our result is {\em optimal} for the exclusion process, and that the weak universality of \eqref{frambuesa} at stationarity is proved for the exclusion process. We point out here that the $s(\cdot)$ and $a(\cdot)$ can have arbitrary, or even infinite range, as long as \eqref{H1} and \eqref{H2} remain in force. For example, the choice $s(z) = \frac{c}{|z|^{1+\beta}}$, $a(z) = \sgn(z) s(z)$  with $\beta>2$, is covered by Theorem \ref{t1}.

Note that the WASEP considered in \cite{BerGia} corresponds to the choice $p_n(1) = \frac{1}{2}$, $p_n(-1) = \frac{1}{2} + \frac{1}{\sqrt n}$. This jump rate does not match \eqref{platano}, since its symmetric part depends on $n$. This difference is irrelevant, since it can be understood as a time change of the process that vanishes in the limit $n \to \infty$. We have chosen \eqref{platano} basically for aesthetic reasons; if instead of \eqref{platano} we take
\[
p_n(z) = s(z) + \gamma_n a(z) \mathbf{1}(z >0), 
\]
assumption \eqref{H1} has to be replaced by
\[
\sum_{z >0} z^2 (s(z)+ |a(z)|) <+\infty,
\]
which we find less pleasant than \eqref{H1}.

The proof of Theorem \ref{t1} follows the method introduced in \cite{GonJar}. The main difficulty in this method is to obtain a {\em second-order Boltzmann-Gibbs principle}, which improves the classical Boltzmann-Gibbs principle (see Chapter 11 of \cite{KipLan} for a discussion on the Boltzmann-Gibbs principle and further references) from the theory of hydrodynamic limits in two ways. First it includes second-order terms and second it gives quantitative bounds on the error terms which are sharp up to constants. For the exclusion process considered in this article, since we are not assuming that $p_n(\cdot)$ has finite range, we have to deal with non-local functions at the level of the Boltzmann-Gibbs principle. This is the main point where new arguments are needed. In fact, the microscopic currents are non-local. A na\"ive cut-off of large jumps is not enough. In Section \ref{trenzinho} we use an integration-by-parts formula to first replace the non-local currents of the system by local terms. This replacement is inspired on Proposition 5 in \cite{GonJarSim}. After this replacement is performed, the proof follows like in \cite{GonJar}.

Here follows an outline of this paper. In Section \ref{s2} we introduce the model and we state our main result, namely Theorem \ref{t1}. In Section \ref{s3} we characterize the limiting points of the sequence of density fluctuation fields   and in Section \ref{tightness} we prove the existence of such limit points by showing that this sequence is tight. In Section \ref{s.proof} we prove Theorem \ref{t1}. In the Appendix we collect some aside results that are needed along the article.

\section{Model and definitions}\label{s2}

\subsection{The exclusion process}

The one-dimensional exclusion process is a Markov process denoted by $\{\eta_t ; t \geq 0\}$ and with space state $\Omega:=\{0,1\}^{\bb Z}$. Its dynamics is defined  in the following way. Let $p: \bb Z \to [0,1]$ be a probability measure which, to avoid degeneracies, satisfies   $p(0)=0$. We call $p(\cdot)$ a {\em transition probability}.  
The value of $\eta_t(x)$ represents the presence or absence of a particle at site $x$ at time $t$, depending on whether $\eta_t(x)=1$ or $\eta_t(x)=0$. The fact that $\eta(x) \in \{0,1\}$ is known as the {\em exclusion rule} which prohibits more than one particle at each site at a given time. We call $\eta_t(x)$ the {\em occupation variables}. Each particle has its own, independent Poissonian clock of rate $1$. Each time a clock rings, the corresponding particle tries to give a jump of a size $z$ chosen according to the transition probability $p(\cdot)$. If the destination site, let us say $x+z$ is empty, then the particle jumps from the departure site $x$; otherwise the particle stays at the site $x$. When the initial number of particles is finite, this process is just a continuous-time Markov chain with bounded rates, and therefore this definition does not pose any problem. When the initial number of particles is infinite, a more careful construction is needed, see Chapters I.3 and VIII of \cite{Lig} for details.
%For each pair $(x,y)$ with $x,y\in\bb Z$ we attach a random clock exponentially distributed with mean $1$. Clocks associated to different pairs are independent. After the ring of a clock of a pair $(x,y)$, if there is a particle at $x$ and no particle at $y$, then the particle jumps from $x$ to $y$ at rate  $p(y-x)$. Conversely, if the site $y$ is occupied and  the site $x$ is empty then the jump occurs from $y$ to $x$ at rate $p(x-y)$. We note that we do not assume here that $p$ is symmetric. Due to the exclusion constraint if both $x$ and $y$ are occupied nothing happens and the clocks restart. 
The dynamics introduced above can be characterized by means of an infinitesimal generator as follows. 
We say that a function $f: \Omega \to \bb R$ is {\em local} if there exists $A \subseteq \bb Z$ finite such that $f$ depends on $\eta$ only through the coordinates $\eta(x)$ with $x\in A$ and in that case we say that the {\em support} of $f$ belongs to $A$. For $f : \Omega \to \bb R$ local we define $Lf: \Omega \to \bb R$ on $\eta \in \Omega$
as
\[
 L f(\eta) = \sum_{x,y \in \bb Z} p(y-x) \eta(x) (1-\eta(y)) \big[ f(\eta^{x,y}) -f(\eta) \big]
\]
where for $x,y \in \bb Z$ and $\eta \in \Omega$  the configuration $\eta^{x,y} \in \Omega$  is defined as
\[
\eta^{x,y}(z) = 
\left\{
\begin{array}{r@{\;;\;}l}
\eta(y) & z =x,\\
\eta(x) & z=y,\\
\eta(z) & z\neq x,y.
\end{array}
\right.
\]
It can be checked that this sum is absolutely convergent since $f$ is local and $p(\cdot)$ is a probability measure. The linear operator $L$ defined in this way turns out to be a pregenerator of a Markov process $\{\eta_t; t \geq 0\}$, see Proposition I.3.2 in \cite{Lig}. 
We call this process the one-dimensional  {\em exclusion process} with transition probability $p(\cdot)$.

For $\rho \in [0,1]$, let $\mu_\rho$ be the Bernoulli product measure in $\Omega$ with density $\rho$. The measures $\{\mu_\rho; \rho \in [0,1]\}$ are invariant under the evolution of $\{\eta_t; t \geq 0\}$. If $p(\cdot)$ is irreducible, then the measures $\{\mu_\rho; \rho \in [0,1]\}$ are also ergodic. Notice that in the dynamics defined above,  particles are neither destroyed nor created, they simply move according to the clocks and the transition probability. This conservation law is reflected in the fact that there are many ergodic measures, (at least) one for each density of particles.

\subsection{Weakly asymmetric exclusion processes}

In this section we introduce a {\em weak asymmetry} in the dynamics that we defined above. For that purpose, let $s: \bb Z \to [0,1]$ be a symmetric transition probability, namely such that $s(z) = s(-z)$ for any $z \in \bb Z$. We assume that
\begin{equation}
\label{H1}
\tag{H1}
s(\cdot) \text{ is irreducible and }  \sigma^2:=\sum_{z \in \bb Z} s(z) z^2  <+\infty.
\end{equation}
In other words, we assume that $s(\cdot)$ has  {\em finite second moment}. Let $a:\bb Z \to \bb R$ be an {\em antisymmetric} function, namely such that $a(z) = - a(-z)$ for any $z \in \bb Z$. We assume that
\begin{equation}
\label{H2}
\tag{H2}
\text{there exists a finite constant $C$ such that } |a(z)| \leq Cs(z) \text{ for any } z \in \bb Z.
\end{equation}

Let $n \in \bb N$ be a scaling parameter and let $-\infty<b<+\infty$ be a non-zero constant. Let $\{\gamma_n; n \in \bb N\}$ be a real-valued sequence such that
\begin{itemize}
\item[i)] $ {\sqrt n} \gamma_n \to b$ as $n \to \infty$, 
\item[ii)] $p_n(z) =: s(z) + \gamma_n a(z)$ is a transition probability.
\end{itemize}

Notice that \eqref{H2} is necessary and sufficient for the existence of such a sequence. Necessity follows from the positivity of $p_n(\cdot)$, and sufficiency follows from taking $\gamma_n = \frac{b}{{\sqrt n}} \mathbf{1}(n \geq \sqrt{Cb})$. Notice as well that \eqref{H1} and \eqref{H2} imply that the series
\begin{equation}\label{asy_average}
m = \sum_{z \in \bb Z} za(z)  
\end{equation}
is absolutely summable. For that purpose notice that 
\[
\sum_{z \in \bb Z} |za(z)|\leq C \sum_{z\in \bb Z}|z| s(z)\leq C\sum_{z\in\bb Z}z^2s(z) <+\infty.  
\]We will also use several times that 
\[
\sum_{z \in \bb Z}z^2 |a(z)|  < +\infty \text{ and, in particular, that } \lim_{z \to \pm\infty} z^2 a(z)  =0.
\]
The {\em weakly asymmetric} exclusion process is the family of processes $\{\eta_t^n;  t \geq 0\}$ generated by 
\[
L_n f(\eta) = n^2 \sum_{x,y \in \bb Z} p_n(y-x) \eta(x) (1-\eta(y)) \big[f(\eta^{x,y})-f(\eta)\big].
\]
In other words, $\{\eta_t^n; t \geq 0\}$  is the exclusion process with transition probability $p_n(\cdot)$ and accelerated by $n^2$. From this acceleration we can already guess that we will be interested in scaling limits of $\{\eta_t^n; t \geq 0\}$ in a {\em diffusive} space-time scaling. Therefore, it makes perfect sense to ask for \eqref{H1}. Once \eqref{H1} is in force, \eqref{H2} is the least we can ask in order to have a well-defined, weakly asymmetric, family of transition probabilities. It is in this sense that we say that 
\eqref{H1} and \eqref{H2} are {\em optimal}.

Notice that the adjoint of $L_n$ with respect to $L^2(\mu_\rho)$ is the generator of the exclusion process with transition probability given at $z\in\bb Z$ by  $p^\ast_n(z) = p_n(-z)$. In particular, the symmetric part of $L_n$ is equal to $n^2 S$, where 
\begin{equation}\label{sym_part_gen}
S f(\eta) = \tfrac{1}{2}\sum_{x,y \in \bb Z} s(y-x) \big[ f(\eta^{x,y})-f(\eta)\big].
\end{equation}

As remarked above, the weakly asymmetric exclusion process is not a process, but rather a family of processes in which the asymmetric part of the transition probability converges to $0$ at speed $n^{-1/2}$.

We  consider the process $\{\eta_t^n; t \geq 0\}$ in a {\em stationary state}, that is, starting from $\mu_\rho$ for some $\rho \in (0,1)$. 
To simplify the exposition we take $\rho =\frac{1}{2}$. For $\rho \neq \frac{1}{2}$, the characteristic velocity of the system is non-zero. The results proven in this article remain in force after performing a Galilean transformation which puts the system in a reference frame on which it has zero characteristic velocity.
We  denote by $\bb P_n$ the law of $\{\eta_t^n; t \geq 0\}$ starting from $\mu_{1/2}$ and by $\bb E_n$ the expectation with respect to $\bb P_n$. We denote by $\<\cdot,\cdot\>$ the inner product in $L^2(\mu_{1/2})$ and for an integrable function $f:\Omega\to\bb R$ we denote by  $\bar f$ its centred version, that is $\bar f:=f -\<f\>$, where  $\<f\>:=\<f,1\>$.

\subsection{The density fluctuation field}

Let $T$ be a finite time horizon.  Let $\mc S(\bb R)$ denote the Schwartz space of test functions  and let $\mc S'(\bb R)$ be the space of distributions, defined as the topological dual of $\mc S(\bb R)$ with respect to the inner-product in $L^2(\bb R)$. Let $\mc D([0,T] ; \mc S'(\bb R))$ (resp.~$\mc C([0,T]; \mc S'(\bb R))$) be the Skorohod space of \cadlag (resp.~continuous) paths in $\mc S'(\bb R)$. For  $n \in \bb N$ we define the process $\{\mc Y_t^n;t \in [0,T]\}$ with trajectories in $\mc D([0,T]; \mc S'(\bb R))$ in the following way. For $t\in  [0,T]$ and $f\in\mc S(\bb R)$,
\[
\mc Y_t^n(f)  = \frac{1}{\sqrt n} \sum_{x \in \bb Z} \bar{\eta}_t^n(x) f \big( \tfrac{x}{n} \big).
\]
The process $\{\mc Y_t^n; t \in [0,T]\}$ is called the  {\em density fluctuation field} associated to the process $\{\eta_t^n; t \in [0,T]\}$.
The main objective of this article is to prove a convergence result for this field. Before stating this result we 
need to define the limiting object $\{\mathcal Y_t; t \in[0,T]\}$.

\subsection{Stationary energy solutions}
\label{Ensol}
In this section we define what it is known as the {\em stationary energy solution} of the stochastic Burgers equation 
\begin{equation}
\label{SBE}
\partial_t \mc Y_t = \tfrac{1}{2}\sigma^2 \Delta \mc Y_t + b m \nabla \mc Y^2_t + \tfrac{1}{2} \sigma \nabla \xi,
\end{equation}
where $\xi$ is a standard space-time white noise, $\sigma^2$ is defined in \eqref{H1}, $b$ is fixed in i) and $m$ is given in \eqref{asy_average}.
 First we introduce the notion of {\em controlled processes}.

%We say that a random distribution $Y$ is a {\em white noise} of variance $\sigma^2$ if for any $f \in \mc S(\bb R)$, $Y(f)$ has a Gaussian law of mean $0$ and variance $\sigma^2 \|f\|_{L^2(\bb R)}^2$.

\begin{definition}[Controlled process \cite{GubPer3}] \label{def:cont_proc}
\quad 

A pair of stochastic processes  $\{(Y_t,A_t); t \in[0,T]\}$ with  trajectories in  $\mc C([0,T]; \mc S'(\bb R))$ is {\em controlled} by the Ornstein-Uhlenbeck equation
\begin{equation}\label{eq:OU}
\partial_t Y_t = \tfrac{1}{2}\sigma^2 \Delta Y_t + \tfrac{\sigma}{ 2} \nabla \xi.
\end{equation}
 if:
\begin{itemize}
\item[i)] For each $t \in [0,T]$, $\mc Y_t$ is a white noise of variance $\frac{1}{4}$,
\item[ii)] for each $f \in \mc S(\bb R)$, the process 
\[
M_t(f) =  Y_t(f) - Y_0(f) - \int_0^t  Y_s(\tfrac{1}{2}\sigma^2  f'') ds -  A_t(f) 
\]
is a Brownian motion of variance $\frac{1}{4} \sigma^2 t\| f'\|_{L^2(\bb R)}^2$,
\item[iii)] for each $f \in \mc S(\bb R)$ the process $\{A_t(f); t \in [0,T]\}$ is a.s.~of  zero quadratic variation,
\item[iv)] for each $T >0$, the pair $\{(Y_{T-t}, -(A_{T-t}-A_T)); t \in [0,T]\}$ satisfies ii).
\end{itemize}
In this case we say that $\{(Y_t,A_t); t \in[0,T]\}$ is a controlled process.
\end{definition}
Above $\| \cdot \|_{L^2(\bb R)}^2$ denotes the usual $L^2$-norm.  Note that if $A_t \equiv 0$, then $\{Y_t; t \in[0,T]\}$ is the stationary solution of the Ornstein-Uhlenbeck equation
\eqref{eq:OU}.
Moreover,  ii) can be understood as a weak formulation of the equation
\[
\partial_t Y_t = \sigma^2 \Delta Y_t + d A_t + \tfrac{1}{ 2} \sigma \nabla \xi.
\]
For controlled processes, it is possible to define the nonlinear term in the stochastic Burgers equation \eqref{SBE} by approximation:

\begin{proposition}
\label{p2}
\quad

Let $\{(Y_t,A_t); t \in[0,T]\}$ be a controlled process and let $\{\iota_\epsilon; \epsilon \in (0,1)\}$ be an approximation of the identity. Then, for any $f \in \mc S(\bb R)$  the limit
\[
 B_t(f) = \lim_{\epsilon \to 0} \int_0^t \int_\bb R \big(Y_s \ast \iota_\epsilon(x)\big)^2 f' (x) dxds
\]
exists and it is independent of the choice of $\iota_\epsilon$. Above $*$ denotes the convolution operator.
\end{proposition}
This proposition has been proved in \cite{GubJar}, \cite{GubPer3}, \cite{GubPer2}. The process $ B_t(f)$ can be understood as a weak version of $\nabla Y_t^2$.

\begin{definition}[Stationary energy solution of \eqref{SBE}]
\quad

A controlled process $\{(Y_t,A_t); t \in[0,T]\}$ is a {\em stationary energy solution} of the stochastic Burgers equation \eqref{SBE} if 
$
A_t(f) = bm  B_t(f)$
a.s.  for all  $f \in \mc S(\bb R)$ and $t\in[0,T]$.
\end{definition}

\begin{proposition} [Uniqueness of stationary energy solution of the stochastic Burgers equation  \cite{DieGubPer}]
\label{p3}
\quad

Any two stationary energy solutions of the stochastic Burgers equation \eqref{SBE} have the same law. 
\end{proposition}

This proposition has been proved in \cite{GubPer2} in the circle and the proof can be extended to the real line $\bb R$ without too many changes \cite{DieGubPer}. The main drawback of Proposition \ref{p3} is that only gives uniqueness for the solution of the stochastic Burgers equation with initial datum given by a white noise. In a seminal paper \cite{Hai1}, Martin Hairer introduced a stronger notion of solution of the stochastic Burgers equation, for which he was able to obtain existence and uniqueness starting from fairly general initial data. This result follows the setting of {\em regularity structures}. In \cite{GubPer3} the authors also obtain existence and uniqueness for solutions of the stochastic Burgers equation with general initial data. Their approach is based on the concept of {\em paraproducts}. Both approaches are not very suitable for proving convergence of microscopic particle systems to the stochastic Burgers equation. 

The notion of stationary energy solution of the stochastic Burgers equation can be understood as a {\em martingale problem} for \eqref{SBE}. It was first  introduced in a slightly different way in \cite{GonJar}, where Proposition \ref{p3} was conjectured to be true and where, assuming Proposition \ref{p3}, the next result was proved for a large class of weakly asymmetric interacting  particle systems with finite-range interactions.

\begin{theorem}
\label{t1} 
The process $\{\mc Y_t^n; t \in [0,T]\}$ converges in law,  as $n\to\infty$, with respect to the Skorohod topology of   $\mc D([0,T]; \mc S'(\bb R))$,
to the \emph{stationary energy solution} of the stochastic Burgers equation given in \eqref{SBE}.
\end{theorem}

\begin{remark} We stress that the hypotheses \eqref{H1} and \eqref{H2} are {\em optimal} in the sense that the very same conditions needed to define the approximations and limiting objects appearing in Theorem \ref{t1} are the ones assumed for the transition probability $p_n(\cdot)$.
\end{remark}

Our main goal will be to prove Theorem \ref{t1}.
We will prove convergence of the sequence $\{\mc Y_t^n; t \in [0,T]\}_{n \in \bb N}$ following the classical strategy of convergence in law, namely we will prove tightness of the sequence with respect to the Skorohod topology, then we will show that any limit point is a stationary energy solution of the stochastic Burgers equation, and then we will invoke uniqueness of such solutions to close the proof of convergence.

\section{The martingale problem}\label{s3}

\subsection{ The associated martingales}

As pointed out in Section \ref{Ensol}, the energy formulation of the stochastic Burgers equation given by  \eqref{SBE} can be understood as a martingale problem. Therefore, martingales will play a prominent role in the proof of Theorem \ref{t1}. In this section we will introduce various martingales and we will explore their relation with the density fluctuation field $\mc Y_t^n$.  By Dynkin's formula, for any function $f \in \mc S(\bb R)$, the process
\begin{equation}\label{eq:mart}
\mc M_t^n(f)  =  \mc Y_t^n(f) =\mc Y_0^n(f) - \int_0^t L_n \mc Y_s^n(f) ds
\end{equation} 
is a martingale and its quadratic variation  is equal to
\begin{equation} \label{quad_variation}
\<\mc M_t^n(f)\> = \int_0^t n \sum_{x,y \in \bb Z} p_n(y-x) \eta_s^n(x) \big(1-\eta_s^n(y)\big) ( f\big(\tfrac{y}{n}\big) -f\big(\tfrac{x}{n}\big))^2 ds.
\end{equation}
We have that
\begin{equation*}
\begin{split}
L_n \eta_t^n(x)
	&=n^2\sum_{y\in\bb Z}\Big(p_n(x-y)\eta_t^n(y)(1-\eta_t^n(x))-p_n(y-x)\eta_t^n(x)(1-\eta_t^n(y))\Big)\\
	&= n^2 \sum_{y \in \bb Z} s(y-x) (\eta_t^n(y) -\eta_t^n(x)) - n^2 \gamma_n \sum_{y \in \bb Z} a(y-x) (\eta_t^n(x) - \eta_t^n(y))^2.
\end{split}
\end{equation*}
Let us rewrite $(\eta_t^n(x)-\eta_t^n(y))^2$ in term of centred variables: 
\[
(\eta_t^n(x)-\eta_t^n(y))^2 = \tfrac{1}{2} - 2\bar{\eta}_t^n(x) \bar{ \eta }_t^n(y).
\]
Therefore, we can rewrite the last term  at the right hand side of \eqref{eq:mart} as
\begin{equation*}
\begin{split}
&\int_0^t\frac{1}{\sqrt n}\sum_{x,y\in\bb Z}\bar{\eta}_s^n(x) n^2 s(y-x)\Big( f\big(\tfrac{y}{n}\big) -f\big(\tfrac{x}{n}\big)\Big)\,ds \\
&\quad -\int_0^t\sqrt{n} \gamma_n \sum_{x,y \in \bb Z}\bar{\eta}_s^n(x)\bar{\eta}_s^n(y) a(y-x) n\Big( f\big(\tfrac{y}{n}\big) -f\big(\tfrac{x}{n}\big)\Big)\, ds.
\end{split}
\end{equation*}
%By the definition of $p_n(\cdot)$ in ii) and  since $s(\cdot)$ is symmetric, last expression can be rewritten as 
%\begin{equation*}
%\begin{split}
%&\int_0^t\frac{n^2}{\sqrt n}\sum_{x,y\in\bb Z}\Big( f\big(\tfrac{y}{n}\big) -f\big(\tfrac{x}{n}\big)\Big)s(y-x)\bar{\eta}_s^n(x)\,ds\\
%-&\int_0^t\frac{n^2}{\sqrt n}\sum_{x,y\in\bb Z}\Big( f\big(\tfrac{y}{n}\big) -f\big(\tfrac{x}{n}\big)\Big)\gamma_n a(y-x)\bar{\eta}_s^n(x)\bar{\eta}_s^n(y)\,ds.
%\end{split}
%\end{equation*}
From this we conclude that 
\[
\int_0^t L_n \mc Y_s^n(f) ds = \int_0^t \mc Y_s^n(\mc S_n f)ds  +  \frac{\gamma_n \sqrt n}{b} \mc A_t^n(f),
\]
where
\begin{equation}\label{def_S}
\mc S_n f\big(\tfrac x n\big) = n^2 \sum_{y \in \bb Z} s(y-x) \Big( f\big(\tfrac{y}{n}\big) -f\big(\tfrac{x}{n}\big)\Big),
\end{equation}
and, 
\[
\mc A_t^n(f) = -b \int_0^t \sum_{x,y \in \bb Z} \bar{\eta}_s^n(x)\bar{\eta}_s^n(y) a(y-x) n\Big( f\big(\tfrac{y}{n}\big) -f\big(\tfrac{x}{n}\big)\Big) \,ds.
\]

By Lemma \ref{numer}, $\mc S_n  f(\frac{x}{n})$ is a discrete approximation of $\frac{1}{2}\sigma^2 f''(\frac{x}{n})$. Therefore, our aim will be to show that the identity
\begin{equation}
\label{martdec}
\mc M_t^n(f) = \mc Y_t^n(f) - \mc Y_0^n(f) - \int_0^t \mc Y_s^n(\mc S_n f) ds - \frac{\gamma_n \sqrt n}{b} \mc A_t^n(f)
\end{equation}
is a discrete version of the energy formulation of the stochastic Burgers equation in ii) of Definition \ref{def:cont_proc}.

\subsection{The microscopic non-linear term of stochastic Burgers equation}
In this section we analyze the microscopic process that corresponds to the non-linear term in the stochastic Burgers equation, namely $\mc A_t^n(f)$ which, after a change of variables, can be rewritten as:
\begin{equation}
\label{pera}
\mc A_t^n(f) = -2b\int_0^t \sum_{\substack{x \in \bb Z\\z >0}} \bar{\eta}_s^n(x) \bar{\eta}_s^n(x+z) a(z) n\Big( f\big(\tfrac{x+z}{n}\big) -f\big(\tfrac{x}{n}\big)\Big) ds.
\end{equation}
Notice that the process $\mc A_t^{n}(f)$ can be written as
$
\int_0^t g(\eta_s^n) ds$
for a properly defined function $g \in L^2(\mu_{1/2})$. We call this kind of processes {\em additive functionals} of the process $\{\eta_t^n; t \in [0,T]\}$. 
Our goal in this section is to write the additive functional \eqref{pera} as a function of the density fluctuation field $\mc Y_t^n$. This will be done in two steps. First, we use an integration by parts formula to write the field $\mc A_t^n(f)$ as the sum of another field which is ``more local'' plus an error term which is negligible in $L^2(\bb P_n)$. This ``more local'' field will only involve terms of the form $\bar \eta(x) \bar \eta(x+1)$. Second, we use the second-order Boltzmann-Gibbs principle, to replace the function $\bar\eta(x)\bar\eta(x+1)$ by a function of the density fluctuation field.

\subsubsection{An integration by parts formula}
\label{trenzinho}
In this section we show that we can replace the product $\bar \eta_t^n(x) \bar \eta_t^n(x+z)$  by the nearest-neighbour product $\bar \eta_t^n(x) \bar \eta_t^n(x+1)$ in \eqref{pera}, at the cost of an error that goes to $0$ in $L^2(\bb P_n)$, as $n \to \infty$. This is the content of the next lemma.

\begin{lemma}

\label{estimate_1}

For any $t\in[0,T]$ and any $f\in\mc S(\bb R)$ it holds that
\begin{equation}
\label{manzana}
\bb E_n\Big[ \Big(\int_0^t \sum_{\substack{x \in \bb Z \\ z > 0}} \big\{ \bar{\eta}_s^n(x) \bar{\eta}_s^n(x+z) - \bar{\eta}_s^n(x) \bar{\eta}_s^n(x+1)\big\} a(z) n\Big( f\big( \tfrac{x+z}{n}\big) - f\big( \tfrac{x}{n}\big)\Big) ds \Big)^2 \Big] \leq \frac{C(f) t}{n}.
\end{equation}
\end{lemma}
\begin{proof}
Let us introduce the notation
\begin{equation}
\label{damasco}
\nabla_{x,x+z}^n f:=n \Big( f\big(\tfrac{x+z}{n}\big) -f\big(\tfrac{x}{n}\big)\Big). 
\end{equation}
From Lemma 2.4 of \cite{KomLanOll}, the expectation in \eqref{manzana} is bounded from above by 
\begin{equation}\label{var_form}
C t \!\!\sup_{h \text{ local}} \!\!\Big\{ 2\sum_{\substack{x \in \bb Z \\ z >0}}  a(z) \nabla_{x,x+z}^n f\<\bar{\eta}(x) \bar \eta(x+z) -\bar \eta(x) \bar \eta(x+1) ,h\> -\< h, -L_n h\> \Big\}.
\end{equation}
Using the change of variables $\eta\to \eta^{x+1,x+z}$, under which the measure $\mu_{1/2}$ is invariant,  we can rewrite the first term of the previous expression as
\begin{equation}
\label{zanahoria}
2\sum_{\substack{x \in \bb Z \\ z >0}}  a(z) \nabla_{x,x+z}^n f\big\< \bar \eta (x) \bar \eta (x+1), h(\eta^{x+1,x+z})-h(\eta)\big\>. 
\end{equation}
Now we use the weighted Cauchy-Schwarz estimate for each $x,z$, to get the bound
\[
\begin{split}
\big|\big\< \bar{\eta}(x) \bar{\eta}(x+1), h(\eta^{x+1,x+z})-h(\eta)\big\>\big|
		&\leq \frac{\beta(x,z)}{2} \int \big(\bar \eta (x) \bar \eta (x+1)\big)^2 d\mu_{1/2}\\
		&\quad + \frac{ I_{x+1,x+z}(h)}{2\beta(x,z)}, 
\end{split}
\]
where for $y,y'\in\bb Z$,
\begin{equation}
\label{naranja}
 I_{y,y'}(h) = \int  \big( h(\eta^{y,y'})-h(\eta)\big)^2 d\mu_{1/2}.
\end{equation}
By Lemma \ref{moving} below, for any $x,z\in\bb Z$ with $z>0$,
\begin{equation}
\label{telesc_arg}
I_{x+1,x+z}(h) \leq (4z-7) \sum_{y=x+1}^{x+z-1} I_{y,y+1}(h).
\end{equation}
%\leq z\mc D (h),
%%\]where \[
%%\mc D(h) = \sum_{x \in \bb Z} \mc D_{x,x+1}(h).
%%\]
Let us define the {\em Dirichlet form} $\mc D(h)$ as
\begin{equation}
\label{lima}
\mc D(h) = \frac{1}{2} \sum_{x \in \bb Z} I_{x,x+1}(h).
\end{equation}
Let us choose $\beta(x,z) = 16\beta |\nabla_{x,x+z}^n f|$. Then \eqref{zanahoria} is bounded above by
\begin{equation}
\label{papa}
\beta \sum_{\substack{x\in \bb Z\\ z >0}} |a(z)| \big|\nabla_{x,x+z}^nf\big|^2 
	+ \frac{1}{8\beta}\sum_{z>0}  |a(z)| z^2 \mc D(h).
\end{equation}
Here we use that $ \int \big(\bar \eta (x) \bar \eta (x+1)\big)^2 d\mu_{1/2}= \frac{1}{16}$, that
\[
\sum_{x \in \bb Z} \sum_{i=1}^{z-1}I_{x+i,x+i+1}(h) = 2(z-1) \mc D(h)
\]
and that $(4z-7)(z-1) \leq 4z^2$.

Now, since the symmetric part of $L_n$ is equal to $n^2 S$, where $S$ is given in \eqref{sym_part_gen}, we have   $\<h,-L_nh\> = n^2 \<h,- Sh\>$.
% 
%$$\<h,-L_nh\> = 
%\<h,-n^2 Sh\>= -n^2\sum_{x, y \in \bb Z} s(y-x) I_{x,y}(h)=-2 n^2\sum_{\substack{x \in \bb Z \\ z >0}}s(z)I_{x,x+z}%(h).
%$$
By Lemma \ref{numer2}, the sum
\[
\frac{1}{n} \sum_{\substack{x \in \bb Z \\ z >0}} |a(z)| |\nabla_{x,x+z}^n f|^2 
\]
is uniformly bounded in $n$.
From Lemma \ref{equiv_dir_form}, taking $\beta=\frac{C_2}{8n^2}\sum_{z>0} z^2|a(z)|$ we get that \eqref{papa} is bounded above by $C(f)/n$, where $C(f)$ only depends on $f$, $a(\cdot)$ and the constant $C_2$ of Lemma \ref{equiv_dir_form}. This ends the proof of the lemma.
\end{proof}

The method described above can be seen as a long-range version of the {\em integrations-by-parts formula} used for example in Section 7.2 of \cite{KipLan}. The particular form presented here was taken from \cite{GonJarSim}.

Finally, we  show the equivalence between Dirichlet forms that was used above.

\begin{lemma}[Moving-particle lemma \cite{Qua}, \cite{DiaS-C}]
\label{moving}
Recall \eqref{naranja}. For any $x, z \in \bb Z$ with $z>0$ and any $h \in L^2(\mu_{1/2})$,
\begin{equation}
\label{uva}
I_{x,x+z}(h) \leq (4z-3) \sum_{y=x}^{x+z-1} I_{y,y+1}(h).
\end{equation}
\end{lemma}
\begin{proof}
This result was proved in \cite{Qua} and \cite{DiaS-C} with a smaller constant in front of the sum. At the expenses of a constant factor, we can give a quicker proof.
First we observe that
\[
\eta^{x,x+z} = \big((\eta^{x,x+1})^{x+1,x+z}\big)^{x,x+1}. 
\]
Then we observe that for any $\theta \in (0,1/2)$,
\[
(a_1+a_2+a_3)^2 \leq \frac{a_1^2+a_2^2}{\theta} + \frac{a_3^2}{1-2\theta}.
\]
Now we write
\[
\begin{split}
h(\eta^{x,x+z}) -h(\eta) 
		&= h\Big( \big( ( \eta^{x,x+1})^{x+1,x+z}\big)^{x,x+1}\Big) -h\big( ( \eta^{x,x+1})^{x+1,x+z}\big)\\
		&\quad + h\big( ( \eta^{x,x+1})^{x+1,x+z}\big) - h( \eta^{x,x+1})\\
		&\quad \quad h( \eta^{x,x+1}) -h(\eta)
\end{split}
\]
to get the bound
\[
I_{x,x+z}(h) \leq \frac{2}{\theta} I_{x,x+1}(h) + \frac{1}{1-2\theta} I_{x+1,x+z}(h).
\]
Now we can proceed in an iterative way: assume that
\[
I_{x+1,x+z}(h) \leq  a(z-1) \sum_{y=x+1}^{z-1} I_{y,y+1}(h).
\]
Then taking $\theta = \frac{2}{4+a(z-1)}$ we see that
\[
I_{x,x+z}(h) \leq (4+a(z-1))\Big( I_{x,x+1}(h) + \frac{1}{a(z-1)}I_{x+1,x+z}(h)\Big).
\]
Since we can trivially take $a(1)=1$, the lemma is proved.
\end{proof}

\begin{lemma}[Equivalence of Dirichlet forms]
\label{equiv_dir_form}

For any $h\in L^2(\mu_{1/2})$, let $\mc D(h)$ be the Dirichlet form defined in \eqref{lima}.
There exist constants $C_1,C_2$ such that
\begin{equation}
\label{durazno}
C_1\<h,- Sh\>\leq \mathcal{D}(h)\leq C_2 \<h,- Sh\>.
\end{equation}
\end{lemma}
\begin{proof}
First we note that
$$\<h,-Sh\>= \tfrac{1}{4}\sum_{x, z \in \bb Z} s(z) \int( h(\eta^{x,x+z})-h(\eta)\big)^2 d\mu_{1/2}=\tfrac{1}{2}\sum_{\substack{x,z \in \bb Z \\ z>0}} s(z) I_{x,x+z}(h).
$$
Now, from \eqref{uva} and by translation invariance,  we get
\begin{equation}
\begin{split}
\<h,-Sh\>
		&\leq \tfrac{1}{2}\sum_{\substack{x, z \in \bb Z\\ z>0 }}  s(z) (4z-3) \sum_{y=x}^{x+z-1} I_{y,y+1}(h)
		\leq \sum_{z\in\mathbb{Z}}(4z-3)z s(z)\mathcal{D}(h)\leq 4 \sigma^2\mathcal{D}(h).
\end{split}
\end{equation}
This proves the estimate on the left-hand side in the statement of the lemma for the choice $C_1=\frac{1}{4\sigma^2}$.

In order to prove the second estimate, let $\{x_0,\dots,x_m\}$ be a sequence in $\bb Z$ such that $x_0 =0$, $x_m=1$ and $s(x_i-x_{i-1})>0$ for any $i \in \{1,\dots,m\}$. This sequence exists by the irreducibility of $s(\cdot)$. Notice that the proof of Lemma \ref{moving} also shows that
\[
I_{x,x+1}(h) \leq (4m-3) \sum_{i=1}^m I_{x_{i-1},x_i}(h). 
\]
Define $z_i = x_i-x_{i-1}$, $i=1,\dots,m$. By translation invariance, we have that
\[
\sum_{x \in \bb Z} I_{x,x+1}(h) \leq (4m-3)\sum_{i=1}^m\sum_{x \in \bb Z} I_{x,x+z_i}(h)
		\leq \frac{(4m+-3)m}{\min_{i} s(z_i)} \sum_{\substack{x, z \in \bb Z \\ z>0}} s(z) I_{x,x+z}(h),
\]
which proves the second estimate in the statement of the lemma for $C_2= \frac{(4m-3)m}{\min_i s(z_i)}$. 
\end{proof}

\begin{remark}
In this article we only need the second estimate in \eqref{durazno}, but for completeness  we included also the first estimate.
\end{remark}

%Now we prove the remaining estimate. Since $s(\cdot)$ is irreducible, we know that for each $x\in\bb Z$, there exists a path of $m_1$ jumps of size $z_1$, $m_2$ jumps of size $z_2$,..., $m_n$ jumps of size $z_n$, where for each $i=1,\cdots, n  $  we have $s(z_i)>0$ and $\sum_{i=1}^nm_iz_i=1$. Let $s_{min}:={\min \{s(z): s(z)> 0\}}$ and note that the size of the path is given by $\sum_{i=1}^n m_i:=m$. 
%Now, by a telescopic argument, the Cauchy-Schwarz inequality and translation invariance we have that
%\begin{equation}
%\begin{split}
%\mc D(h)=\sum_{x\in\bb Z} \int  \big( h&(\eta^{x,x+1})-h(\eta)\big)^2 d\mu_{1/2}\leq \sum_{x\in\mathbb{Z}}m \sum_{i=1}^n m_i\int (h(\eta^{x,x+z_i})-h(\eta))^2 d\mu_{1/2}
%\\&\leq \sum_{x\in\mathbb{Z}}\frac{m}{\min \{s(z): s(z)> 0\}} \sum_{i=1}^{n}m_is(z_i)\int (h(\eta^{x,x+z_i})-h(\eta))^2d\mu_{1/2}
%\\ &\leq C_2 \<h, -Sh\>, \textrm{for} \, C_2= {m^2}/{s_{min}}.
%\end{split}
%\end{equation}
%This ends the proof.

\subsubsection{The second-order Boltzmann-Gibbs principle}
Notice that by \eqref{pera} and \eqref{damasco} we have that 
\[
\mc A_t^n(f) = -2b\int_0^t  \sum_{\substack{x \in \bb Z\\z >0}} \bar{\eta}_s^n(x) \bar{\eta}_s^n(x+z) a(z) n \nabla_{x,x+z}^n f ds.
\]
Define 
\begin{equation}\label{der_f}
\widetilde \nabla_x^n f =2n \sum_{z >0} a(z)\nabla_{x,x+z}^n f.
\end{equation}
By Lemma \ref{numer3},
\[
\lim_{n \to \infty} \tfrac{1}{n} \sum_{x \in \bb Z} \big( \widetilde \nabla_x^n f - m f'\big(\tfrac{x}{n}\big)\big)^2 =0.
\]
Let us define the process $\{\hat{\mc A}_t^n(f); t \in [0,T]\}$ as
\[
\hat{\mc A}_t^n(f) = -b\int_0^t \sum_{x \in \bb Z} \bar \eta_s^n(x) \bar \eta_s^n(x+1) \widetilde \nabla_x^n f ds.
\]
In Lemma \ref{estimate_1} above, what we have showed is that the processes $\mc A_t^n(f)$ and $\hat{\mc A}_t^n(f)$ are asymptotically equivalent in $L^2(\bb P_n)$. The following proposition shows that the process $\hat{ \mc A}_t^n(f)$ can be approximated by a functional of the density of particles:

\begin{proposition}
\label{2BG}
There exists a finite constant $C$ such that
\begin{equation}
\label{2BGbound}
\bb E_n \Big[ \Big( \int_0^t \sum_{x \in \bb Z} \big\{ \bar \eta_s^n(x) \bar \eta_s^n(x+1) - \psi_x^{\ell}(\eta_s^n) \big\} \widetilde \nabla_x^n f ds \Big)^2 \Big] \leq \frac{C \ell t}{n^2} \sum_{x \in \bb Z} \big(\widetilde \nabla_x^n f \big)^2.
\end{equation}
for any $t \geq 0$, any $f \in \mc S(\bb R)$ and any $\ell,n \in \bb N$, 
where for $\eta\in\Omega$,  
\[
\psi_x^\ell(\eta):=\tfrac{\ell}{\ell-1} \big( \eta^\ell(x) - \tfrac{1}{2}\big)^2 -\tfrac{1}{4(\ell-1)}
\]
and 
\[
\eta^\ell(x):=\frac 1\ell \sum_{y=x}^{x+\ell-1}\eta(y).
\]
\end{proposition}
Last proposition is known as the {\em second-order Boltzmann-Gibbs principle} and it was introduced in \cite{GonJar}. The version presented here is a particular case of Corollary 2 in \cite{GonJar}, so we refer to that article for a proof. 

Define $\tilde{\psi}_x^\ell(\eta):=\Big(\eta^\ell(x)-\frac 1 2 \Big)^2$.
By the Cauchy-Schwarz inequality we see that 
 \begin{equation}
 \label{frutilla}
\bb E_n \Big[ \Big( \int_0^t \sum_{x \in \bb Z} \Big(\psi_x^{\ell}(\eta_s^n) - \tilde {\psi}_x^\ell(\eta_s^n) \Big) \widetilde \nabla_x^n f ds \Big)^2 \Big] \leq \frac{Ct}{\ell^2} \sum_{x \in \bb Z} \big(\widetilde \nabla_x^n f \big)^2,
\end{equation}
Let $\epsilon >0$ be fixed. Note that taking $\ell=\epsilon n$, the right hand-side of \eqref{2BGbound} is of order $t\epsilon$ and the right-hand side of \eqref{frutilla} is of order $t^2/(\epsilon^2 n)$. 
Moreover, observe that for the choice $\ell=\epsilon n$ we have that 
\begin{equation}
\label{mora}
\sum_{x \in \bb Z}  \tilde{\psi}_x^{\epsilon n}(\eta_s^n)\widetilde \nabla_x^n f = \frac 1 n \sum_{x \in \bb Z} \Big(\mc Y_s^n ( \epsilon^{-1}\mathbf{1}_{[\frac{x}{n},\frac{x}{n}+\epsilon)})\big( \tfrac{x}{n}\big)\Big)^2 \widetilde \nabla_x^n f.
\end{equation}
Define $\iota_\epsilon=\epsilon^{-1}\mathbf{1}_{[0,\epsilon)}$ and let $\ast$ denote the convolution operator.
The right-hand side of \eqref{mora} can be written as
\begin{equation}
\label{cuadraprox}
 \frac 1 n \sum_{x \in \bb Z} \Big(\mc Y_s^n \ast \iota_\epsilon\big( \tfrac{x}{n}\big)\Big)^2 \widetilde \nabla_x^n f,
\end{equation}
which is a function of the density field $\mc Y_s^n$. We conclude, at least heuristically, that $\mc A_t^n(f)$ looks like an approximation of the field $m \nabla (\mc Y_s^n)^2$ acting on the test function $f$. For later reference, we define for $\epsilon >0$, $n \in \bb N$, $t \in [0,T]$ and $f \in \mc S(\bb R)$
\[
\mc A_t^{n,\epsilon}(f) = \int_0^t \sum_{x \in \bb Z} \tilde{\psi}_x^{\epsilon n} (\eta_s^n)\widetilde \nabla_x^n f ds.
\]

\section{Tightness}\label {tightness}

In this section we prove tightness of the sequence $\{\mc Y_t^n; t \in [0,T]\}_{n \in \bb N}$. We will use Mitoma's criterion:

\begin{proposition}[Mitoma's criterion \cite{Mit}]
\label{Mit}
\quad

Fix $T >0$. A sequence $\{Y_t^n; t \in [0,T]\}_{n \in \bb N}$ of distribution-valued processes is tight in $\mc D([0,T]; \mc S'(\bb R))$ if and only if the sequence $\{\mc Y_t^n(f); t \in [0,T]\}_{n \in \bb N}$ of real-valued processes is tight in $\mc D([0,T]; \bb R)$ for any $f \in \mc S(\bb R)$. 
\end{proposition}
Under this criterion we are left to check tightness for the real-valued processes $\{\mc Y_t^n(f); t \in [0,T]\}_{n \in \bb N}$ for any $f \in \mc S(\bb R)$. By \eqref{martdec}, it is enough to show tightness for the sequence of random variables $\{\mc Y_0^n(f)\}_{n \in \bb N}$ and for each one of the sequences of processes
\begin{equation}
\label{processes}
\{\mc M_t^n(f); t \in [0,T]\}_{n \in \bb N},\; \{\mc A_t^n(f); t \in [0,T]\}_{n \in \bb N}, \; \{\textstyle{\int}_0^t \mc Y_s^n(\mc S_n f) ds; t \in [0,T]\}_{n \in \bb N}.
\end{equation}

\subsection{Tightness at the initial time}
We start with the simplest one which is the initial data $\{\mc Y_0^n(f)\}_{n \in \bb N}$. In fact,  a simple computation  shows that the characteristic function of $\mc Y_0^n(f)$ satisfies 
\[
\lim_{n \to \infty} \log \bb E_n\big[ e^{i\theta \mc Y_0^n(f)}\big] = -\frac{1}{8}\theta \|f\|_{L^2(\bb R)}^2.
\]
Therefore, $\{\mc Y_0^n(f)\}_{n \in \bb N}$  converges to a Gaussian law of mean $0$ and variance $\frac{1}{4} \|f\|_{L^2(\bb R)}^2$ and, in particular, it is tight.

\subsection{Convergence of the martingales}
Now we turn into the proof of tightness for the martingales. In fact we are going to prove that 
the sequence of martingales  $\{\mc M_t^n(f); t \in [0,T]\}_{n \in \bb N}$ converges in law, as $n \to \infty$ to a Brownian motion of quadratic variation $\frac{1}{4} \sigma^2 \|\nabla f\|_{L^2(\bb R)}^2 t$, with respect to the uniform topology of $\mc D([0,T]; \bb R)$. We will use the following criterion, taken from Theorem 2.1 in \cite{Whi}:

\begin{proposition}
\label{martconv}
Let $\{M_t^n; t \in [0,T]\}_{n \in \bb N}$ be a sequence of square-integrable martingales. Let $\{\<M_t^n\>; t \in [0,T]\}$ denote the quadratic variation of $\{M_t^n; t \in [0,T]\}$ and define the maximal jump size of $M_t^n$ as the process
\[
\Delta M_T^n = \sup_{s \leq t} \big|M_t^n-M_{t-}^n\big|.
\]
Assume that
\begin{itemize}
\item[i)] the quadratic variation $\{\<M_t^n\>; t \in [0,T]\}$ is a continuous process,
\item[ii)] the maximal jump satisfies
\[
\lim_{n \to \infty} E\big[\big(\Delta M_T^n \big)^2 \big] =0,
\]
\item[iii)] the quadratic variation processes $\{\<M_t^n\>; t \in [0,T]\}_{n \in \bb N}$ converge in law to the deterministic path $\{\sigma^2 t; t \in  [0,T]\}$.
\end{itemize}
Then $\{M_t^n; t \in [0,T]\}_{n \in \bb N}$ converges in law to a Brownian motion of quadratic variation $\sigma^2 t$.
\end{proposition}
Recall from \eqref{quad_variation} that 
\begin{equation} \label{q_v}
\<\mc M_t^n(f)\> 
		= \int_0^t n \sum_{x,y \in \bb Z} p_n(y-x) \eta_s^n(x) \big(1-\eta_s^n(y)\big) ( f\big(\tfrac{y}{n}\big) -f\big(\tfrac{x}{n}\big))^2 ds.
\end{equation}
Therefore, condition i) of  Proposition \ref{martconv} is automatically satisfied.
Now, we show that the jumps of $\{\mc M_t^n(f); t \in [0,T]\}_{n\in\mathbb{N}}$ are asymptotically negligible. First, from \eqref{martdec}  we see that
\[
\sup_{t \in [0,T]} \big| \mc M_t^n(f)-\mc M_{t-}^n(f)\big| =\sup_{t \in [0,T]} \big| \mc Y_t^n(f)-\mc Y_{t-}^n(f)\big|,
\] 
since all the other terms in \eqref{martdec} are continuous in time.
Now, since at each time $t$ there is at most one jump, which changes the value of $\eta$ at two positions, this supremum is bounded by $\frac{2\|f\|_\infty}{ \sqrt n}$, which vanishes as $n\to\infty$.
Therefore, we are left to prove convergence in law of the process $\{\<\mc M_t^n(f)\>; t \in [0,T]\}_{n\in\mathbb{N}}$. 
Note that $ s(z)=(p(z)+p(-z))/2\geq p(z)/2$ for any $z \in \bb Z$. Therefore, from \eqref{q_v} we get that 
\begin{equation}
\label{lips}
\tfrac{d}{dt} \<\mc M_t^n(f) \> \leq \tfrac{2}{n} \sum_{x,y \in \bb Z} s(y-x) \big( \nabla_{x,y}^n f\big)^2.
\end{equation}
By Lemma \ref{numer2}, the right-hand side of this expression  converges to $2\sigma^2 \|\nabla f\|_{L^2(\bb R)}^2$ as $n\to\infty$, and in particular it is uniformly bounded in $n$ by a constant $c(f)$ that does not depend on $t$. In particular, the process $\{\<\mc M_t^n(f)\>; t \in [0,T]\}_{n\in\mathbb{N}}$ is uniformly Lipschitz in $t \in [0,T]$ and $n \in \bb N$. Therefore, the sequence $\{\<\mc M_t^n(f)\>; t \in [0,T]\}_{n \in \bb N}$ is equicontinuous and, in particular, tight.

The convergence of finite-dimensional distributions can be shown as follows. First, recall  \eqref{q_v}. We can write the sum in \eqref{q_v} as twice its half and in one of the halves we can make the exchange of variables $x \leftrightarrow y$ to get that
\[
 \bb E_n \Big[ \<\mc M_t^n(f)\>\Big] =  \frac{t}{4n} \sum_{x,y \in \bb Z} s(y-x) \big( \nabla_{x,y}^n f\big)^2 \xrightarrow{n \to \infty} \tfrac{1}{4} \sigma^2 t \| f'\|_{L^2(\bb R)}^2.
\]
Last convergence above follows from Lemma \ref{numer2}. Then we prove the convergence of the quadratic variation. From \eqref{q_v} and by centering, we get that
\[
\begin{split}
\bb E_n\Big[ \Big(\<\mc M_t^n(f)\> -& \bb E_n \Big[ \<\mc M_t^n(f)\>\Big]\Big)^2\Big]\\
		&= \bb E_n\Big[\Big(\int_0^t \tfrac{1}{n}\sum_{x,y}p_n(y-x)\big(\eta^n_s(x)(1-\eta^n_s(y))-\chi(\rho)\big)\big(\nabla_{x,y}^n f\big)^2\, ds\Big)^2\Big]\\
				&\leq 2 \bb E_n\Big[\Big(\int_0^t -\tfrac{1}{n}\sum_{x,y}p_n(y-x)\bar {\eta}^n_s(x)\bar{\eta}^n_s(y)\big(\nabla_{x,y}^n f\big)^2\, ds\Big)^2\Big]
				\\&+2\bb E_n\Big[\Big(\int_0^t \tfrac{1}{n} \sum_{x,y}\frac{p_n(y-x)}{2}\Big(\bar {\eta}^n_s(x)-\bar{\eta}^n_s(y)\Big)\big(\nabla_{x,y}^n f\big)^2\, ds\Big)^2\Big] 
\end{split}
\]
By a change of variables, the second term above can be written as
\[
\bb E_n\Big[\Big(\int_0^t \tfrac{1}{n}\sum_{x,y}a(y-x)\bar {\eta}^n_s(x)\big(\nabla_{x,y}^n f\big)^2\, ds\Big)^2\Big] 
\]
and by the Cauchy-Schwarz inequality and independence, it is bounded from above by 
$$ \frac{C t^2}{n^2} \sum_{x \in \bb Z} \Big( \sum_{y \in \bb Z} a(y-x) \big( \nabla_{x,y}^n f\big)^2 \Big)^2\leq\frac{c(f)t ^2}{n}.
		$$
Similarly, by a change of variables, the remaining term can be written as 
\[ \bb E_n\Big[\Big(\int_0^t -\tfrac{1}{n}\sum_{x,y}s_n(y-x)\bar {\eta}^n_s(x)\bar{\eta}^n_s(y)\big(\nabla_{x,y}^n f\big)^2\, ds\Big)^2\Big]				
\]
and by the Cauchy-Schwarz inequality plus independence,  it is bounded from above by
\[
 \frac{C t^2}{n^2} \sum_{x,y \in \bb Z} s(y-x)^2 \big(\nabla_{x,y}^n f\big)^4 \leq \frac{c(f)t ^2}{n}.
\]

From this, we conclude that the martingales $\{\mc M_t^n(f); t \in [0,T]\}_{n\in\mathbb{N}}$ converge in law, as $n \to \infty$, to a Brownian motion of quadratic variation $\frac{1}{4} \sigma^2 t\| f'\|_{L^2(\bb R)}^2$.

\subsection{Tightness of the integral terms }

Let us now prove tightness for the sequence $\{\int_0^t \mc Y_s^n(\mc S_n f) ds; t \in [0,T]\}_{n \in \bb N}$. 
In this case we use the so-called {\em Kolmogorov-Centsov criterion}:

\begin{proposition}[Kolmogorov-Centsov's criterion]
\label{KolCen}
A sequence $\{x_t^n; t \in [0,T]\}_{n \in \bb N}$ of continuous, real-valued, stochastic processes is tight with respect to the uniform topology of $\mc C([0,T]; \bb R)$ if the sequence of real-valued random variables $\{x_0^n\}_{n \in \bb N}$ is tight and there are positive constants $K,\alpha,\beta$ such that
\[
E[|x_t^n-x_s^n|^\alpha] \leq K|t-s|^{1+\beta}
\]
for any $s,t \in [0,T]$ and any $n \in \bb N$.
\end{proposition} 
In the case  of integral processes, that is, processes of the form $\{\int_0^t x_s^n ds; t \in [0,T]\}$, a simple consequence of Proposition \ref{KolCen} is the following:
\begin{proposition}
\label{inttig}
A sequence of processes of the form $\{\int_0^t x_s^n ds; t \in [0,T]\}_{n \in \bb N}$ is tight, with respect to the uniform topology of  $\mc C([0,T]; \bb R)$, if
\[
\limsup_{n \to \infty} \sup_{0 \leq t \leq T} E[(x_t^n)^2] <+\infty.
\]
\end{proposition}
By independence, a simple computation shows that 
\[
\bb E_n \big[ \mc Y_s^n(\mc S_n f)^2 \big] = \tfrac{1}{4n} \sum_{x \in \bb Z} \big( \mc S_n f \big(\tfrac{x}{n}\big)\big)^2.
\]
We can bound the previous expression from above by
\begin{equation}
\begin{split}
&\tfrac{1}{2n} \sum_{x \in \bb Z} \big( \mc S_n f \big(\tfrac{x}{n}\big)-\frac{1}{2}\sigma^2 \Big(f'' \big(\tfrac{x}{n}\big)\Big)^2+ 
c(f)\\
\leq &\tfrac{1}{2n} \sup_{x\in\bb Z}\Big| \mc S_n f \big(\tfrac{x}{n}\big)-\frac{1}{2}\sigma^2f'' \big(\tfrac{x}{n}\big)\Big| \sum_{x \in \bb Z}\Big| \mc S_n f \big(\tfrac{x}{n}\big)-\frac{1}{2}\sigma^2f'' \big(\tfrac{x}{n}\big)\Big|+ 
c(f)
\end{split}
\end{equation}
By Lemma \ref{numer}, the term on the left hand side of the previous expression vanishes as $n\to\infty$, from where we conclude that  this sum is uniformly bounded in $n$. Therefore, Proposition \ref{inttig} shows that the sequence $\{\int_0^t \mc Y_s^n( \mc S_n f) ds ; t \in [0,T]\}_{n \in \bb N}$ is tight.

\subsection{Tightness of the non-linear term}

In order to prove tightness of $\{\mc Y_t^n(f); t \in [0,T]\}_{n \in \bb N}$, the most demanding term is $\{\mc A_t^n(f); t \in [0,T]\}_{n \in \bb N}$. By Lemma \ref{estimate_1}, if $\mc R_t^n(f) = \mc A_t^n(f) - \hat{ \mc A}_t^n(f)$, then $\bb{E}_{n}[(\mc R_t^n(f)^2]\leq c(f)t/n$, so that $\mc R_t^n(f)$ converges to $0$ in $L^2(\bb P_n)$, as $n\to\infty$. But this is not enough to get convergence to $0$ as {\em a process}. If we were able to prove tightness of $\{\mc R_t^n(f); t \in [0,T]\}_{n \in \bb N}$, then the convergence, as a process, would follow. We note however that since  the variance bound we  have so far is linear in $t$, we cannot use the  Kolmogorov-Centsov's criterion (Proposition \eqref{KolCen}). Neverthless,  using the Cauchy-Schwarz inequality we can obtain a crude estimate on the variance of $\mc R_t^n(f)$:
\[
\begin{split}
\bb E_n\big[ \mc R_t^n(f)^2\big] 
		&\leq t^2 \int \Big(\sum_{\substack{x \in \bb Z \\ z >0}} a(z) \big( \bar \eta(x) \bar \eta (x+z) - \bar \eta(x) \bar \eta(x+1) \big) \nabla_{x,x+z}^n f \Big)^2 d \mu_{1/2}\\
		&\leq \frac{t^2}{8} \Big(\sum_{\substack{x \in \bb Z \\ z >0}} a(z)^2 \big( \nabla_{x,x+z}^n f\big)^2 + \sum_{x \in \bb Z} \Big(\sum_{z > 0} a(z) \nabla_{x,x+z}^n f\Big)^2\Big)\leq c(f) t^2 n,
\end{split}
\]
for some constant $c(f)$. Therefore, now we have the bound
\[
\bb E_n \big[ \mc R_t^n(f)^2\big] \leq c(f) \min\{ t n^{-1}, t^2 n\}.
\]
One can easily check that these two bounds imply that, for any $\delta>0$,
\[
\bb E_n\big[ \big(\mc R_t^n(f)-\mc R_s^n(f)\big)^2\big] \leq c(f) \frac{|t-s|^{3/2-\delta}}{n^\delta}.
\]
 In particular, the Kolmogorov-Centsov's criterion now applies, from where we conclude that the sequence $\{\mc R_t^n(f); t \in [0,T]\}_{n \in \bb N}$ is tight. To finish, we are left to prove tightness of $\{\hat{\mc A}_t^n(f); t \in [0,T]\}_{n \in \bb N}$. By Proposition \ref{2BG},
\[
\bb E_n\big[\big( \hat{\mc A}_t^n(f) - \mc A_t^{n,\epsilon}(f)\big)^2\big] \leq c(f) t \epsilon,
\]
and by the Cauchy-Schwarz inequality we have that
\[
\bb E_n\big[ \mc A_t^{n,\epsilon}(f)^2\big] \leq \frac{t^2}{8\epsilon n} \sum_{x \in \bb Z} \big( \widetilde \nabla_x^n f\big)^2 \leq c(f) t^2 \epsilon^{-1}.
\]
Taking $\epsilon = \sqrt t$ we get the bound
\begin{equation}
\label{3/4bound}
\bb E_n\big[ \big( \hat{\mc A}_t^n(f) - \hat{\mc A}_s^n(f)\big)^2\big] \leq c(f) |t-s|^{3/2}.
\end{equation}
Therefore, by Proposition \ref{KolCen}, the sequence $\{\hat{\mc A}_t^n(f); t \in [0,T]\}_{n \in \bb N}$ is tight. From this we conclude that  the sequence $\{\mc A_t^n; t \in [0,T]\}_{n \in \bb N}$ is tight.

\section{Proof of Theorem \ref{t1}}\label{s.proof}

Since, from the computations of Section \ref{tightness}, we know that the sequence $\{\mc Y_t^n; t \in [0,T]\}_{n \in \bb N}$ is tight, then it has limit points. Let $\{\mc Y_t; t \in [0,T]\}$ be one of those limit points and let $n'$ be a subsequence such that $\mc Y_t^{n'} \to \mc Y_t$. If we can prove that $\{\mc Y_t; t \in [0,T]\}$ is a stationary energy solution of the stochastic Burgers equation \eqref{SBE}, then  convergence in law of $\{\mc Y_t^n;t \in [0,T]\}$ to $\{\mc Y_t; t \in [0,T]\}$ follows from the uniqueness in law of such solutions. Before proving that $\{\mc Y_t; t \in [0,T]\}$ is a stationary energy solution of \eqref{SBE}, we need to find a process $\{\mc A_t; t \in [0,T]\}$ such that the pair $\{(\mc Y_t, \mc A_t); t \in [0,T]\}$ is a controlled process according to Definition \ref{def:cont_proc}. Let us now check  all the conditions of Definition \ref{def:cont_proc}.

We start with i). A simple computation, which was in fact done in Section \ref{tightness}, shows that $\{\mc Y_0^n(f)\}_{n \in \bb N}$ converges to a Gaussian random variable of mean $0$ and variance $\frac{1}{4}$. Therefore, $\mc Y_0$ is a white noise of variance $\frac{1}{4}$. By stationarity of the measure $\mu_{1/2}$, the same is true for $\mc Y_t$, for any $t \in [0,T]$. This shows i).

For ii), the idea now is to take  the limit $n'\to\infty$ in the equation  \eqref{martdec}. Without loss of generality, we can assume that each one of the processes
\begin{equation}
\label{processes_0}
\{\mc M_t^n(f); t \in [0,T]\}_{n \in \bb N},\; \{\mc A_t^n(f); t \in [0,T]\}_{n \in \bb N}, \; \{\textstyle{\int}_0^t \mc Y_s^n(\mc S_n f) ds; t \in [0,T]\}_{n \in \bb N}
\end{equation} 
converges, as $n'\to\infty$,  to some limiting process.  Let $\{\mc A_t(f); t \in [0,T]\}$ be the limit of $\{\mc A_t^{n'}(f); t \in [0,T]\}$. From Section \ref{tightness} we  know that $\mc M_t(f)$ given by
\[
\mc Y_t(f) - \mc Y_0(f) - \int_0^t \mc Y_s(\sigma^2 f'') ds - b \mc A_t(f)
\]
is a Brownian motion of variance $\frac{1}{4} \sigma^2t \| f'\|_{L^2(\bb R)}^2$. Therefore, the pair $\{(\mc Y_t, \mc A_t); t \in [0,T]\}$ satisfies condition ii) of the definition of controlled processes.

Now we prove condition iii). First we observe that  since $\{\mc R_t^{n'}(f); t \in [0,T]\}$ converges to $0$, as $n'\to\infty$,  $\mc A_t(f)$ is also the limit of $\hat{\mc A}_t^{n'}(f)$. Since the bound \eqref{3/4bound} is an upper bound which is uniform in $n$, it is preserved by convergence in law. Therefore,
\begin{equation}\label{energy_estimate}
\bb E\big[ \big( \mc A_t(f)-\mc A_s(f)\big)^2\big] \leq c(f) |t-s|^{3/2}.
\end{equation}
%
%
%from where we conclude that $\{\mc A_t(f); t \in [0,T]\}$ has zero quadratic variation and condition iii) of the definition of controlled processes. 
%This notion of solutions is actually equivalent to the one given in \cite{GubPer}, except that test functions are defined on $\bb R$ here, instead of the finite torus $\bb T$. More precisely, the energy estimate \eqref{eq:energy_estimate} implies that : 
From this, one can conclude that  the process $\{\mc A_t(H); t\in[0,T]\}$  is {\em a.s.~}of zero quadratic variation in the sense of Russo and Vallois \cite{RusVal}. Indeed, from \eqref{energy_estimate}, Fatou's lemma and Fubini's Theorem, we have that
\begin{align*}
\bb E \bigg[ \lim_{\varepsilon \to 0} \int_0^t \frac{1}{\varepsilon} \big(\mc A_{s+\varepsilon}(H)-\mc A_s(H)\big)^2\; ds  \bigg] & \leq \lim_{\varepsilon \to 0} c(f) T \sqrt {\varepsilon} = 0.
\end{align*}
This proves condition iii).
Condition iv) of the definition of controlled processes follows from the following observation: the law of the process $\{\eta_{T-t}^n; t \in [0,T]\}$ corresponds to an exclusion process with transition probability $p_n^\ast(z) = p_n^\ast(-z)$ and initial measure $\mu_{1/2}$.
Therefore, we can repeat the computations that we did above for the process $\{\eta_{T-t}^n; t \in [0,T]\}$, from where iv) follows. 
Putting all together, we have just proved that $\{(\mc Y_t, \mc A_t); t \in [0,T]\}$ is a controlled process. In particular,  from Proposition \ref{p2}, the limit 
\[
 \lim_{\epsilon \to 0} \int_0^t \int \big(\mc Y_s \ast \iota_\epsilon(x)\big)^2 f' (x) dxds
\]
exists. Taking the limit in \eqref{cuadraprox} along the subsequence $n'$ we see that
\[
\lim_{n' \to \infty} \mc A_t^{n',\epsilon}(f) = m \int_0^t \int \big(\mc Y_s \ast \iota_\epsilon(x)\big)^2 f'(x) dx ds.
\]
Therefore, on one hand
\[
\lim_{\epsilon \to 0} \lim_{n' \to \infty} \mc A_t^{n',\epsilon}(f) = m \mc B_t(f)
\quad \textrm{and}\quad 
\lim_{n' \to \infty} \hat{\mc A}_t^{n'}(f) = \mc A_t(f);
\]
and on the other hand
\[
\bb E_n\big[\big( \hat{\mc A}_t^n(f) - \mc A_t^{n,\epsilon}(f)\big)^2\big] \leq c(f) t \epsilon.
\]
In particular, the limits in $\epsilon$ and $n'$ are interchangeable, from where we conclude that
$\mc A_t(f) = m \mc B_t(f)
$ {\em a.s}
which implies that 
$\{(\mc Y_t, \mc A_t); t \in [0,T]\}$ is an energy solution of the stochastic Burgers equation \eqref{SBE}. This  proves Theorem \ref{t1}.

\section{Acknowledgments}
The authors thank the French Ministry of Education through the grant ANR (EDNHS).  P. G. thanks  FCT/Portugal for support through the project 
UID/MAT/04459/2013.

\appendix

\section{Discrete approximations}\label{ap1}

\begin{lemma}
\label{numer}
Let $f \in \mc S(\bb R)$. Then,
\begin{equation}
\begin{split}
(a)& \quad \lim_{n \to \infty} \sup_{x \in \bb Z} \Big|\mc S_n f \big(\tfrac{x}{n}\big) - \tfrac{1}{2}\sigma^2  f''\big( \tfrac{x}{n}\big) \Big| =0,\\
(b)&\quad \lim_{n \to \infty} \tfrac{1}{n} \sum_{x \in \bb Z} \Big|\mc S_n f \big(\tfrac{x}{n}\big) - \tfrac{1}{2}\sigma^2  f''\big( \tfrac{x}{n}\big) \Big| =0. 
\end{split}
\end{equation}
\end{lemma}
\begin{proof}
We start by showing $(a)$. Recall \eqref{def_S} and, for $\epsilon >$ and $n\in\bb N$ split there the sum  into the sets $\{|z| \leq \epsilon n\}$ and $\{|z| > \epsilon n\}$.  Then,
\begin{equation}
\label{mandarina}
\begin{split}
\Big|\mc S_n f \big(\tfrac{x}{n}\big) - \tfrac{1}{2}\sigma^2  f''\big( \tfrac{x}{n}\big) \Big|
		&\leq\Big|\sum_{|z|\leq \epsilon n} n^2 s(z)\Big(f\big(\tfrac{x+z}{n}\big) -f\big( \tfrac{x}{n}\big)  -\tfrac{z^2}{2n^2}f''\big(\tfrac{x}{n}\big)\Big)\Big|\\
		&\quad +\Big|\sum_{|z|> \epsilon n} n^2 s(z)\Big(f\big(\tfrac{x+z}{n}\big) -f\big( \tfrac{x}{n}\big)  -\tfrac{z^2}{2n^2}f''\big(\tfrac{x}{n}\big)\Big)\Big|
\end{split}
\end{equation}
Using Taylor's formula, we have that 
\[
\Big| f\big(\tfrac{x+z}{n}\big) -f\big( \tfrac{x}{n}\big) - \tfrac{z}{n} f'\big(\tfrac{x}{n}\big) -\tfrac{z^2}{2n^2}f''\big(\tfrac{x}{n}\big)\Big| \leq \tfrac{z^3}{6 n^3} \|f'''\|_\infty.
\]  
The term $\frac{z}{n} f'(\frac{x}{n})$ can be added into the first sum in \eqref{mandarina} using the symmetry of $s(\cdot)$. Therefore, the first sum in \eqref{mandarina} is bounded above by $c(f)\sigma^2 \epsilon$.
In the second sum in \eqref{mandarina}, cancellations do not play a role. We get the bound
\[
\begin{split}
\Big|\sum_{|z|> \epsilon n} n^2 s(z)\Big(f\big(\tfrac{x+z}{n}\big) -f\big( \tfrac{x}{n}\big)  -\tfrac{z^2}{2n^2}f''\big(\tfrac{x}{n}\big)\Big)\Big|
		&\leq \sum_{|z| >\epsilon n} s(z) \big\{ 2n^2\|f\|_\infty + \tfrac{1}{2} \|f''\|_\infty z^2\big\}\\
		&\leq \big( \tfrac{2}{\epsilon^2} \|f\|_\infty + \tfrac{1}{2} \|f''\|_\infty \big) \sum_{|z| > \epsilon n} s(z) z^2.
\end{split}
\]
Putting all these estimates together we get the bound
\[
\begin{split}
\Big| \mc S_n f\big(\tfrac{x}{n}\big) - \tfrac{1}{2} \sigma^2 f''\big(\tfrac{x}{n}\big)\Big| 
		&\leq c(f)\sigma^2\epsilon+C\Big(\frac{\|f\|_{\infty}}{\epsilon} +\|f''\|_\infty\big)  \sum_{|z| > \epsilon n} z^2s(z).
\end{split}
\]
Taking $n \to \infty$ and then $\epsilon \to 0$, $(a)$ is proved. 

Now we prove $(b)$. Since we already know that $\mc S_nf(\frac{x}{n})-\frac{1}{2} \sigma^2 f''(\frac{x}{n})$ converges uniformly to $0$, as $n\to\infty$,  if we can show that the tails of $\mc S_n f(\frac{x}{n})$ are uniformly summable in $n$, $(b)$ follows from standard arguments. Fix $x \in \bb Z$. Since $s(\cdot)$ is symmetric we can write
\begin{equation}
\label{chirimoya}
n^2 \sum_{z \in \bb Z} s(z) \Big( f\big(\tfrac{x+z}{n}\big) -f\big(\tfrac{x}{n}\big)\Big) 
		= n^2 \sum_{z \in \bb Z} s(z) \Big( f\big(\tfrac{x+z}{n}\big) -f\big(\tfrac{x}{n}\big) - \tfrac{z}{n} f'\big(\tfrac{x}{n}\big)\mathbf{1}(|z| \leq \tfrac{1}{2}|x|)\Big).
\end{equation}
Now define the function $F:\bb R \to \bb R$ as
\[
F(u) = \sup_{|y-u| \leq u/2} \big| f''(y)\big|
\]
and split the sum on the right-hand side of \eqref{chirimoya}
into the sets $\{|z| \leq \frac{1}{2}|x|\}$ and $\{|z\ > \frac{1}{2} |x|\}$. 
In the set $\{|z| \leq \frac{1}{2} |x|\}$ we use the bound
\[
\big| f\big(\tfrac{x+z}{n}\big) - f\big( \tfrac{x}{n}\big) -\tfrac{z}{n} f'\big( \tfrac{x}{n}\big)\big| 
	\leq \tfrac{z^2}{2n^2} \!\!\!\sup_{|y-x| \leq \tfrac{1}{2}|x|} \!\!\!\big| f''\big(\tfrac{y}{n}\big)\big| \leq \tfrac{z^2}{2n^2}F\big(\tfrac{x}{n} \big)
\]
to get the bound
\[
n^2 \sum_{|z| \leq \frac{1}{2} x} s(z) \big| f\big(\tfrac{x+z}{n} \big) -f\big( \tfrac{x}{n}\big)\big| \leq \tfrac{1}{2} \sigma^2F\big(\tfrac{x}{n} \big).
\]
In the set $\{|z| > \tfrac{1}{2}|x|\}$ the points $\frac{x}{n}$ and $\frac{x+z}{n}$ are far away, so cancellations are not useful. Therefore, we bound the difference $|f(\frac{x+z}{n}) - f(\frac{x}{n})|$ by $2\|f\|_\infty$ to get the bound
\[
n^2 \sum_{|z| > \tfrac{1}{2}|x|} s(z) \big| f\big(\tfrac{x+z}{n}\big) - f\big( \tfrac{x}{n}\big)\big| 
		\leq \frac{8\|f\|_\infty \sigma^2}{|x/n|^2}.
\]
Since $f \in \mc S(\bb R)$, the function $F(u)$ decays to $0$ at infinity faster than any power of $u$. Moreover, we have also seen that $\mc S f (x/n)$ is uniformly bounded, from where we conclude  that  there is a constant $C(f)$ such that
$\big| \mc S_n  f\big(\tfrac{x}{n}\big)\big| \leq \frac{C(f)}{1+(x/n)^2}$. The limit in b) then follows from a) and standard truncation over sets of the form $\{|z| \leq Mn\}$.
\end{proof}

This lemma has the following consequence:

\begin{lemma}
\label{numer2}
Let $f \in \mc S(\bb R)$. Then,
\[
\lim_{n \to \infty} \tfrac{1}{n} \sum_{x,y \in \bb Z} s(y-x) \big( \nabla_{x,y}^n f \big)^2 = \sigma^2\| f'\|_{L^2(\bb R)}^2.
\]
\end{lemma}
\begin{proof}
We start by observing that, by writing the sum in $\mc Sf(x/n)$ as twice its half and in one of the terms by making a change of variables and using  the symmetry of $s(\cdot)$, we get that
\[
\tfrac{1}{n} \sum_{x \in \bb Z} f\big(\tfrac{x}{n}\big) \big(\mc S_n f\big(\tfrac{x}{n}\big)\big)= 
-\tfrac{1}{2n} \sum_{x,z \in \bb Z} n^2s(z) \Big(f\Big(\frac{x+z}{n}\Big)-f\Big(\frac{x}{n}\Big)\Big) =-\tfrac{1}{2n} \sum_{x,y \in \bb Z} s(y-x) \big( \nabla_{x,y}^n f \big)^2.
\]
Now, to prove the result we do the following:
\begin{equation*}
\begin{split}
\tfrac{1}{n} \sum_{x,y \in \bb Z} s(y-x) \big( \nabla_{x,y}^n f \big)^2=\tfrac{2}{n} \sum_{x \in \bb Z} f\big(\tfrac{x}{n}\big) \big(-\mc S_n f\big(\tfrac{x}{n}\big)+\frac{1}{2}\sigma^2 f''\Big(\frac{x}{n}\Big)\big)-\frac{\sigma^2}{n}\sum_{x\in\mathbb{Z}}f\big(\tfrac{x}{n}\big)f''\Big(\frac{x}{n}\Big).
\end{split}
\end{equation*}
From $(a)$ in Lemma \ref{numer} the first  term on the right hand side of the previous expression vanishes as $n\to\infty$. The second term, by a summation by parts, converges, as $n\to\infty$,  to  $\sigma^2\| f'\|_{L^2(\bb R)}^2$, which finishes the proof.
\end{proof}

The computation used to prove part b) of Lemma \ref{numer} can also be used to prove the following lemma:

\begin{lemma}
\label{numer3}
Let $f \in \mc S(\bb R)$ and  recall \eqref{der_f}. 
Then,
\[
\lim_{n \to \infty} \tfrac{1}{n} \sum_{x \in \bb Z} \big( \widetilde{\nabla}_x^n f -m f'\big(\tfrac{x}{n}\big)\big)^2 =0.
\]
\end{lemma}
\begin{proof}
Notice that
\[
\widetilde{\nabla}_x^n f -m f'\big(\tfrac{x}{n}\big) = 2n \sum_{z > 0} a(z) \Big(f\big(\tfrac{x+z}{n} \big) -f\big(\tfrac{x}{n}\big) -\tfrac{z}{n}f'\big(\tfrac{x}{n}\big)\Big).
\]
Fix $L \in \bb N$. Using Taylor's formula we see that
\[
\begin{split}
\Big|2n \sum_{z = 1}^L a(z) \Big(f\big(\tfrac{x+z}{n} \big) -f\big(\tfrac{x}{n}\big) -\tfrac{z}{n}f'\big(\tfrac{x}{n}\big)\Big)\Big| 
		&\leq 2n \sum_{z=1}^L  \frac{|a(z)|z^2}{n^2} \sup_{0 \leq y \leq L} \big| f''\big(\tfrac{x+y}{n}\big)\big| \\
		&\leq \frac{C}{n} \sup_{0 \leq y \leq L} \big| f''\big(\tfrac{x+y}{n}\big)\big|.
\end{split}
\]
On the other hand,
\[
\Big|2n \sum_{z > L} a(z) \Big(f\big(\tfrac{x+z}{n} \big) -f\big(\tfrac{x}{n}\big)\Big)\Big| 
		\leq 4 n \|f\|_\infty \sum_{z > L}  \frac{|a(z)| z^2}{L^2} \leq \frac{C(f) n}{L^2}
\]
and
\[
\Big|2n \sum_{z > L} a(z) \tfrac{z}{n} f'\big(\tfrac{x}{n}\big)\Big|  \leq 2\big|f'\big(\tfrac{x}{n}\big)\big| \sum_{z > L} \tfrac{z^2}{L} \leq \frac{Cf'\big(\tfrac{x}{n}\big)}{L}.
\]
Taking $L=n$ for $|x|\leq 2n$ and $L = \frac{x}{2}$ for $|x| \geq 2n$ we see that 
\[
\big| \widetilde{\nabla}_x^n f - m f'\big(\tfrac{x}{n}\big)\big| \leq \tfrac{C}{n} F\big(\tfrac{x}{n}\big),
\]
where
\[
F(u) = \min\big\{ (1+u^2)^{-1/2}, \sup_{|y| \leq \frac{1}{2}|u|} \big|F''(y)\big|\big\}.
\]
Since $f \in \mc S(\bb R)$, the sum
\[
\tfrac{1}{n} \sum_{x \in \bb Z} F\big(\tfrac{x}{n}\big)^2
\]
is uniformly bounded in $n$, which proves the lemma.
\end{proof}
\bibliographystyle{plain}

\end{document}